\documentclass[11pt,lenq]{amsart}
\usepackage{amsmath}
\usepackage{amscd}
\usepackage{amssymb}
\usepackage{amsbsy}
\usepackage{amsfonts}
\usepackage{color}
\usepackage{latexsym}
\usepackage{graphics}
\usepackage{graphicx}
\usepackage{amsmath,amscd,latexsym}
\usepackage{array}
\usepackage{enumerate}
\theoremstyle{plain}
\newtheorem{theorem}{Theorem}[section]

\newtheorem{lemma}[theorem]{Lemma}

\newtheorem{main theorem}[theorem]{Main Theorem}

\newtheorem{question}[theorem]{Question}

\theoremstyle{definition}   
\newtheorem{remark}[theorem]{Remark}

\newlength\savewidth

\newcommand{\BS}{\mathrm{BS}}

\newcommand{\svert}{\,|\,}


\begin{document}
\title[A non-Hopfian ascending HNN-extension of a finitely presented Hopfian group]
{A non-Hopfian ascending HNN-extension of a finitely presented Hopfian group}

\author{Jan Kim}
\address{Institute of Mathematical Sciences \\
	Ewha Womans University \\
	52 Ewhayeodae-gil, Seodaemun-gu, Seoul 03760, Republic of Korea}
\curraddr{Department of Mathematical Sciences \\
    KAIST, 291 Daehak-ro, Yuseong-gu, 34141 Daejeon, South Korea}
\email{jankim@kaist.ac.kr}		

\author{Junseok Kim}
\address{Department of Mathematics \\
        Technion - Israel Institute of Technology \\
        Technion City, Haifa, Israel, 3200003}
\email{jsk8818@campus.technion.ac.il}

\author{Yoonjin Lee}
\address{Department of Mathematics, Ewha Womans University \\
	52 Ewhayeodae-gil, Seodaemun-gu, Seoul 03760, Republic of Korea;
	Korea Institute for Advanced Study, 85 Hoegi-ro, Dongdaemun-gu, Seoul 02455, Republic of Korea}
\email{yoonjinl@ewha.ac.kr}

\thanks{Jan Kim was supported by the Basic Science Research Program through the National Research Foundation of Korea(NRF) funded by the Ministry of Education (RS-2021-NR060138), Junseok Kim was supported by Israel Science Foundation Grant 1576/23 (PI Nir Lazarovich), and Yoonjin Lee was supported by the Basic Science Research Program through the National Research Foundation of Korea (NRF) funded by the Ministry of Education (NRF-2019R1A6A1A11051177) and by the National Research Foundation of Korea (NRF) grant funded by the Korea government (RS-2022-NR069041).}


\begin{abstract}
We discover the first example of a non-Hopfian ascending HNN-extension of a finitely presented Hopfian group by providing an explicit construction. 
This result addresses an analogous question to the one posed by Sapir and Wise, which asks whether there is a non-residually finite ascending HNN-extension of a finitely presented residually finite group. Such an analogy is motivated by Mal'cev's result that every finitely generated residually finite group is Hopfian.
\end{abstract}

\maketitle

\section{Introduction}
\label{sec:introduction}		

A group $G$ is called {\em Hopfian} if every epimorphism $G \rightarrow G$ is an automorphism. Equivalently, a group $G$ is Hopfian if and only if it is not isomorphic to any proper quotient of itself.
In 1931, Hopf~\cite{Hopf} asked whether a finitely generated group can be isomorphic to a proper quotient of itself.
This question was answered in the affirmative by Neumann~\cite{Neumann}, who constructed the first finitely generated non-Hopfian group in 1950.
In the following year, Higman~\cite{Higman} constructed the first finitely presented non-Hopfian group $\langle a, s, t \svert s^{-1}as=a^2, \ t^{-1}at=a^2 \rangle$, which has the form of a multiple HNN-extension.

An important source of the Hopf property is residual finiteness.
A group $G$ is called {\em residually finite} if for every $g \in G \setminus \{1\}$, there is a finite group $P$ and an epimorphism $\psi : G \rightarrow P$ so that $\psi(g) \neq 1$.
For finitely generated groups, many groups classified by certain properties are residually finite, for example, abelian groups, free groups, nilpotent groups~\cite{Hirsch}, and linear groups~\cite{Malcev}. 
On the other hand, Mal'cev~\cite{Malcev} proved that every finitely generated residually finite group is Hopfian. Consequently, finitely generated non-Hopfian groups provide a natural source of non-residually finite groups. In particular, the non-Hopfian groups constructed by Neumann and Higman discussed above are not residually finite.
A more focused setting is that of non-residually finite one-relator groups: Baumslag and Solitar~\cite{Baumslag_Solitar} showed that the group $\BS(2,3)=\langle a, b \svert b^{-1}a^2b=a^3 \rangle$  is such an example, called a \textit{Baumslag-Solitar group}. 

Residual finiteness is preserved under many constructions. For instance, Borisov and Sapir~\cite{Borisov_Sapir} proved that every ascending HNN-extension of a finitely generated linear group is residually finite, and Hsu and Wise~\cite{Hsu_Wise} showed that every ascending HNN-extension of a polycyclic-by-finite group is residually finite.
On the other hand, Sapir and Wise~\cite{Sapir_Wise} demonstrated that ascending HNN-extensions of finitely generated residually finite groups need not be residually finite.
We note that the base group constructed in Sapir-Wise's paper~\cite{Sapir_Wise} is not finitely presented. Indeed, the base group is the \textit{double} of a free group along an infinitely generated subgroup. That is, it is the amalgamated free product of two free groups $F_1$ and $F_2$ of the same rank, obtained by identifying an infinitely generated subgroup of $F_1$ with its copy in $F_2$. The infinite number of relators comes from the infinitely generated amalgamated subgroup, and in fact, the second homology of the base group is infinitely generated.

Modifying their construction to obtain a finitely presented residually finite base group satisfying Sapir-Wise's recipe seems to be complicated.
Therefore, a natural question arises, as posed by Sapir and Wise~\cite[Problem 3.9]{Sapir_Wise}: 

\bigskip
{\em Is there a non-residually finite ascending HNN extension of a finitely presented residually finite group?}

\bigskip

Sapir and Wise~\cite{Sapir_Wise} showed that the ascending HNN-extension of the finitely generated residually finite group that they constructed is non-residually finite by proving that it is non-Hopfian.
More generally, they provided a recipe to construct non-Hopfian ascending HNN-extensions~\cite[Lemma~3.1]{Sapir_Wise}.
Given Mal'cev's result and its use by Sapir and Wise to construct a non-residually finite group by proving it is non-Hopfian, we approach the original question by addressing the following related question.

\bigskip
{\em Is there a non-Hopfian ascending HNN-extension of a finitely presented Hopfian group?}

\bigskip
We answer this question by providing an explicit construction as follows.

\begin{theorem}
\label{thm:main_theorem}
Let $G$ be a group with the following presentation:
\[
\begin{aligned}
	G=\langle a, b, s, t \svert &b^{-1}ab=a, \ s^{-1}a^2s=a^4, \\
	&t^{-1}at=a^2, \ t^{-1}bt=b, \ t^{-1}st=s^{-1}bs^2 \rangle.
\end{aligned}
\]
Then $G$ is a non-Hopfian ascending HNN-extension of a Hopfian group
\[
H=\langle a, b, s \svert b^{-1}ab=a, \ s^{-1}a^2s=a^4 \rangle.
\]
\end{theorem}

\;
\;

We note that there is a well-known open problem in geometric group theory, posed by Gromov~\cite{Gromov}, which asks whether every hyperbolic group is residually finite. Although this question remains open, recently it has been established that every hyperbolic group is Hopfian~\cite{Fujiwara_Sela, Reinfeldt_Weidmann}. Sapir and Wise~\cite[Conjecture~1.3]{Sapir_Wise} conjectured that every ascending HNN-extension with a hyperbolic base group is also Hopfian. In this regard, Theorem~\ref{thm:main_theorem} provides an indication that one might attempt to construct a non-Hopfian ascending HNN-extension of a hyperbolic group.

On the other hand, the Hopfian group $H$ given in Theorem~\ref{thm:main_theorem} served as a peripheral subgroup of a non-Hopfian relatively hyperbolic group constructed in~\cite[Example~1.5]{Kim_Kyzy_Lee}. The construction in~\cite[Example~1.5]{Kim_Kyzy_Lee} is a non-Hopfian HNN-extension of a Hopfian group but not an ascending HNN-extension, which distinguishes it from our work.

In this paper, we construct a non-Hopfian ascending HNN-extension of a finitely presented Hopfian group by providing an explicit construction in Theorem~\ref{thm:main_theorem}.
In Section~\ref{sec:proof}, we prove Theorem~\ref{thm:main_theorem}. The proof utilizes the lemma of Sapir and Wise~\cite[Lemma~3.1]{Sapir_Wise} regarding the construction of non-Hopfian ascending HNN-extensions. In Section~\ref{sec:non_relatively_hyperbolic}, in connection with Gromov's question, we examine whether our result yields the relative hyperbolicity. In detail, we show that certain non-Hopfian ascending HNN-extensions with specific Hopfian base groups, including the one from Theorem~\ref{thm:main_theorem} and that of Sapir and Wise, are all non-relatively hyperbolic.

\; 

\section*{Acknowledgement} 

The authors are grateful to Professor Daniel Wise for his positive feedback on this work and for encouraging us to write up our results. The authors also thank Ravi Tomar for pointing out Remark~\ref{rmk}.

\;

\section{Proof of Theorem 1.1}
\label{sec:proof}

The following recipe is to construct non-Hopfian ascending HNN-extensions, as introduced by Sapir and Wise.

\begin{lemma}{\cite[Lemma~3.1]{Sapir_Wise}}
\label{lem:Sapir_Wise}
Let $H$ be a group and let $\phi : H \to H$ be an injective homomorphism. Suppose there exists a homomorphism $\psi : H \to H$ which is not injective, such that $\phi$ and $\psi$ commute, i.e., $\phi \circ \psi = \psi \circ \phi$, and $\phi(H) \subset \psi(H)$. Let $G$ be the ascending HNN-extension of $\phi$, that is,
\[
G = \langle H, t \mid t^{-1}ht = \phi(h), \ \text{for all} \ h \in H \rangle.
\]
Then by extending $\psi : H \to H$ to $\tilde{\psi} : G \to G$ via $\tilde{\psi}(t) = t$, one obtains a surjective but non-injective endomorphism of $G$, proving $G$ is non-Hopfian.
\end{lemma}

\bigskip

\begin{remark}
In general, to construct $G$ as a non-Hopfian ascending HNN-extension of a Hopfian group $H$ using Lemma~\ref{lem:Sapir_Wise}, $\phi$ must not be surjective, since otherwise, $\psi$ is also surjective, contradicting the assumption that $H$ is Hopfian. Moreover, the existence of an injective and non-surjective homomorphism $\phi$ implies that $H$ is not co-Hopfian. Furthermore, since any ascending HNN-extension of finitely generated linear group is residually finite (and hence Hopfian)~\cite{Borisov_Sapir}, we must consider $H$ as a non-linear group to construct non-Hopfian ascending HNN-extensions.
\end{remark}

\bigskip

We now prove Theorem~\ref{thm:main_theorem}.

\begin{proof}[Proof of Theorem~\ref{thm:main_theorem}]
Note that 
\begin{equation}
\label{equ:H}
H=\langle a, b, s \svert b^{-1}ab=a, \ s^{-1}a^2s=a^4 \rangle
\end{equation}
can be regarded as an HNN-extension where the base group is $C=\langle a,b\rangle(\cong\mathbb{Z}^2)$, the stable letter is $s$, and the associated subgroups are $\langle a^2\rangle$ and $\langle a^4\rangle$.
The group $H$ is Hopfian by~\cite[Theorem~3]{Andreadakis} with $k=0$, $p=2$ and $q=4$.
We aim to show that $G$ is non-Hopfian by applying Lemma~\ref{lem:Sapir_Wise}.

Let $f$ be a unique homomorphism from a free group with basis $\{a, b, s\}$ to $H$ induced by the mapping:
\[
a \mapsto a^2, \quad b \mapsto b \quad \text{and} \quad s \mapsto s^{-1}bs^2.
\]
Then every defining relator in the presentation~(\ref{equ:H}) is sent to the identity in $H$. Hence, $f$ induces an endomorphism $\phi$ of $H$.

To apply Lemma~\ref{lem:Sapir_Wise}, we first show that $\phi$ is injective.
Let $h\in H$ be a non-identity element. Using a normal form of the HNN-extension, $h$ can be written as a reduced word
\[h\equiv h_0s^{\varepsilon_1}h_1s^{\varepsilon_2}\cdots h_{n-1}s^{\varepsilon_n}h_n,\]
where each $h_i \equiv a^{j_i} b^{k_i} \in C$ with $j_i, k_i \in \mathbb{Z}$ for $i = 0, \dots, n$ and $\varepsilon_i=\pm 1$ for $i = 1, \dots, n$
(Here, the symbol ``$\equiv$’’ denotes {\em letter-by-letter equality}.)
without subword of the form $s^{-1} k s$ or $s k' s^{-1}$, where $k \in \langle a^2 \rangle$ and $k' \in \langle a^4 \rangle$ in the expression of $h$.
To show that $\phi(h)\neq 1$, by Britton's Lemma, it is enough to verify that the expression $\phi(h)$, obtained by applying $\phi$ to each letter of $h \equiv h_0 s^{\varepsilon_1} h_1 s^{\varepsilon_2} \cdots s^{\varepsilon_n} h_n$, can be rewritten in an expression that contains no subword of the form $s^{-1} a^{2m} s$ or $s a^{4m} s^{-1}$ for any $m \in \mathbb{Z} \setminus \{0\}$.

If $h_i \equiv a^{j_i}b^{k_i}$ with $j_i, k_i \in \mathbb{Z}$ and $k_i \neq 0$ for every $i=1, \dots, n-1$, then $\phi(h)$ has no subword of the form $s^{-1}a^{2m}s$ or $sa^{4m}s^{-1}$ for $m \in \mathbb{Z} \setminus \{0\}$.
Otherwise, suppose that there exists $p \in \{1, \dots, n-1\}$ such that $h_p \equiv a^{j_p}$ for some $j_p \in \mathbb{Z} \setminus \{0\}$.
Then the subword $s^{\varepsilon_p}h_p s^{\varepsilon_{p+1}}$ of $h$ is one of the following:
\[
s^{-1}a^L s^{-1}, \quad sa^Ls, \quad s^{-1}a^Ms, \quad \text{and} \quad s a^Ns^{-1},
\]
where $L \in \mathbb{Z} \setminus \{0\}$, $a^M \notin \langle a^2 \rangle$ and $a^N \notin \langle a^4 \rangle$.
Applying $\phi$ to each of these subwords, we obtain 
\[
\begin{aligned}
\phi(s^{-1}a^L s^{-1}) &\equiv s^{-2}b^{-1}sa^{2L}s^{-2}b^{-1}s, &
\phi(sa^Ls) &\equiv s^{-1}bs^2a^{2L}s^{-1}bs^2, \\
\phi(s^{-1}a^Ms) &\equiv s^{-2}b^{-1}sa^{2M}s^{-1}bs^2, &
\phi(sa^Ns^{-1}) &\equiv s^{-1}bs^2a^{2N}s^{-2}b^{-1}s.
\end{aligned}
\]
Then each of $\phi(s^{-1}a^Ls^{-1})$, $\phi(sa^Ls)$, and $\phi(sa^Ns^{-1})$ may contain a subword of the form $sa^{4m}s^{-1}$ for some nonzero integer $m$.
If $\phi(s^{-1}a^L s^{-1})$ or $\phi(sa^Ls)$ contains such a subword, then replacing $sa^{2L}s^{-1}$ by $a^L$ eliminates all subwords of the form $sa^{4m}s^{-1}$ with $m \ne 0$ from $\phi(s^{-1}a^L s^{-1})$ and $\phi(sa^Ls)$. Also, if $\phi(sa^Ns^{-1})$ contains a subword of the form $sa^{4m}s^{-1}$ for some nonzero integer $m$, then, since $a^N \notin \langle a^4 \rangle$, replacing $sa^{2N}s^{-1}$ by $a^N$ ensures that the resulting word contains no subword of the form $sa^{4m}s^{-1}$ with $m \ne 0$.
On the other hand, since $M$ is not divisible by $2$, $\phi(s^{-1}a^Ms) \equiv s^{-2}b^{-1}sa^{2M}s^{-1}bs^2$ cannot contain a subword of the form $sa^{4m}s^{-1}$ for any nonzero integer $m$.
Therefore, we have $\phi(h) \neq 1$, and hence, $\phi$ is injective.

Now, let $g$ be a unique homomorphism from a free group with basis $\{a, b, s\}$ to $H$ induced by the mapping:
\[
a \mapsto a^2, \quad b \mapsto b \quad \text{and} \quad s \mapsto s.
\]
Then all defining relators of the presentation~(\ref{equ:H}) are mapped to the identity element in $H$.
Since $\psi(s^{-1}asa^{-2})=1$ but $s^{-1}asa^{-2} \neq 1$ in $H$ by Britton's Lemma, $\psi$ is not injective. 
Moreover, $\psi$ is not surjective since $a \notin \textrm{im}\psi$.
Furthermore, it is straightforward to check that $\phi \circ \psi=\psi \circ \phi$ and that $\phi(H) \subset \psi(H)$.
By Lemma~\ref{lem:Sapir_Wise}, the mapping torus of $\phi$, that is,
\[
G=\langle H, t \svert t^{-1}at=a^2, \ t^{-1}bt=b, \ t^{-1}st=s^{-1}bs^2 \rangle
\]
is non-Hopfian. Indeed, as aforementioned in Lemma~\ref{lem:Sapir_Wise}, the homomorphism $\tilde{\psi}:G\to G$, obtained by extending $\psi:H\to H$ with $\tilde{\psi}(t)=t$, is surjective but not injective.
\end{proof}

\;

\section{Non-relative hyperbolicity of certain ascending HNN-extensions}
\label{sec:non_relatively_hyperbolic}

There is a longstanding open problem posed by Gromov~\cite{Gromov} asking whether every hyperbolic group is residually finite. Related progress has shown that every hyperbolic group is Hopfian~\cite{Fujiwara_Sela, Reinfeldt_Weidmann}. In~\cite[Conjecture~1.3]{Sapir_Wise}, Sapir and Wise conjectured that for hyperbolic base groups, every ascending HNN-extension is Hopfian. In this context, the existence of a non-Hopfian ascending HNN-extension of a finitely presented Hopfian group, which we construct in Theorem~\ref{thm:main_theorem}, offers an indication that one might attempt to construct a non-Hopfian ascending HNN-extension of a hyperbolic group. Accordingly, one can ask the following question.

\bigskip

\begin{question}
Is it possible to construct a non-Hopfian ascending HNN-extension of a hyperbolic group using Lemma~\ref{lem:Sapir_Wise}?
\end{question}

\bigskip

Note that Gromov's question of whether every hyperbolic group is residually finite is equivalent to Osin's question of whether every relatively hyperbolic group with respect to residually finite peripheral subgroups is residually finite~\cite{Osin, Osin_2}. 
Relatedly, in~\cite[Example~1.5]{Kim_Kyzy_Lee}, the Hopfian group $H$ appearing in our Theorem~\ref{thm:main_theorem} was used as a peripheral subgroup in the construction of a non-Hopfian relatively hyperbolic group.
The construction in~\cite[Example~1.5]{Kim_Kyzy_Lee} is a non-Hopfian HNN-extension of a Hopfian group but not an ascending HNN-extension, which distinguishes it from our work.
This motivates the natural question of whether our construction yields a relatively hyperbolic group.
In fact, a relatively hyperbolic group cannot be obtained using the group $H$ from Theorem~\ref{thm:main_theorem} as a base group of an ascending HNN-extension. Moreover, the ascending HNN-extension of a Hopfian group considered in~\cite[Theorem~3]{Andreadakis} also does not yield a relatively hyperbolic group, as we show below.

\;

\begin{theorem}\label{thm:nrh1}
For integers $p,q,k$ with $p>0$, let
\[
H(p,q,k)=\langle a,b,s \svert b^{-1}ab=a, \ s^{-1}a^ps=a^q b^k\rangle ,
\]
which is the group introduced in~\cite[Theorems 2 and 3]{Andreadakis},
and let
\[
\phi:H(p,q,k) \to H(p,q,k)
\]
be an arbitrary injective homomorphism. 
Then the ascending HNN-extension 
\[
G = \langle H(p,q,k), t \mid t^{-1}ht = \phi(h) , \ \text{for all} \ h \in H(p,q,k) \rangle
\]
is not relatively hyperbolic.
\end{theorem}

\bigskip

\begin{proof}
Suppose, for a contradiction, that $G$ is relatively hyperbolic with respect to a finite collection of proper finitely generated subgroups. If $a$ is a hyperbolic element, since $a$ has infinite order, $a$ is contained in a unique maximal elementary subgroup $E(a)=\{g\in G:g^{-1}a^n g=a^{\pm n} \,\,\text{for some $n\in\mathbb{N}$}\}$ (\cite[Theorem 4.3]{Osin_3}). However, since $[a,b]=1$, we obtain $b\in E(a)$, and so $\langle a,b\rangle\cong\mathbb{Z}^2$ is a subgroup of $E(a)$, which is impossible. Hence, $a$ must be a parabolic element.

Suppose $a\in P$ for some parabolic subgroup $P$. By the almost malnormality of parabolic subgroups (\cite[Theorem 1.4]{Osin}), $a\in P\cap b^{-1}Pb$ and the fact that $a$ has infinite order imply that $b\in P$. Similarly, since $a^p\in P\cap sPs^{-1}$, we have $s\in P$. Note that $H(p,q,k)$ is generated by $a,s$, and $t$, so we obtain $H\subset P$. Furthermore, $t$ is contained in $P$ since $a\in P\cap tPt^{-1}$ and $\phi(a)\in H(p,q,k) \subset P$. Therefore, we have $P=G$, and thus $G$ is not relatively hyperbolic.   
\end{proof}

\begin{remark}\label{rmk}
    There is a dynamical analogue of relatively hyperbolic groups, namely the notion of a \emph{non-elementary convergence group}.
    For non-elementary convergence groups, the following properties hold: (1) every $\mathbb{Z}^2$ subgroup is contained in a parabolic subgroup, and (2) maximal parabolic subgroups are almost malnormal (see \cite[Lemmas~8.7 and~8.8]{Tomar}).
    Therefore, the group $G$ in Theorem~\ref{thm:nrh1} is not a non-elementary convergence group.
\end{remark}

The group $T$ given by the following presentation is the Sapir–Wise's~\cite{Sapir_Wise} non-Hopfian ascending HNN-extension of a finitely generated residually finite group:
\[
\begin{aligned}
T=\langle a, b, c, A, B, C, t \svert &
c^nb^{2^n}a^{2^n}b^{-2^n}c^{-n}=C^nB^{2^n}A^{2^n}B^{-2^n}C^{-n}: n \ge 0, \\
&tat^{-1}=ca^2c^{-1}, \ tbt^{-1}=cb^2c^{-1}, \ tct^{-1}=c, \\
&tAt^{-1}=CA^2C^{-1}, \ tBt^{-1}=CB^2C^{-1}, \ tCt^{-1}=C\rangle.
\end{aligned}
\]
They showed that $T$ is an ascending HNN-extension of the amalgamated product of two free groups with an infinitely generated amalgamated subgroup. Note that $T$ is torsion-free since free groups are torsion-free.

Following the non-relative hyperbolicity of the ascending HNN-extension in Theorem~\ref{thm:nrh1}, we also investigate whether Sapir-Wise's group $T$ is relatively hyperbolic. We prove that the group $T$ is not relatively hyperbolic, using an argument similar to that in the proof of Theorem~\ref{thm:nrh1}.

\bigskip

\begin{theorem}
The group $T$ is not relatively hyperbolic.
\end{theorem}

\begin{proof}
    Suppose that $T$ is relatively hyperbolic with respect to a finite collection of proper finitely generated subgroups. The same argument in the proof of Theorem~\ref{thm:nrh1} shows that $t$, $c$, and $C$ are parabolic elements contained in the same parabolic subgroup, say $P$.
    
    Note that $a$ cannot be a hyperbolic element since $a$ is conjugate to $a^2$ (\cite[Corollary 1.15]{Osin}), so $a$ must be contained in some parabolic subgroup $P'$. Since $a\in P'\cap t^{-1}cP'c^{-1}t$, by the almost malnormality of parabolic subgroups, we have $t^{-1}c\in P'$. However, $t^{-1}c$ also belongs to $P$, and the intersection of two different parabolic subgroups must be finite, so we have $P=P'$. Similarly, one can see that $P$ contains the other generators $b$, $A$, and $B$, and hence, we conclude that $P=T$, a contradiction.
\end{proof}

\;

As the previously mentioned ascending HNN-extensions are not relatively hyperbolic, we end this paper with the following question:

\bigskip

\begin{question}
Does every non-Hopfian ascending HNN-extension of a Hopfian group fail to be relatively hyperbolic?
\end{question}

\medskip
\bibliographystyle{abbrv} 
\bibliography{reference}

@article{Andreadakis,
  title={Residual finiteness and Hopficity of certain HNN extensions},
  author={Andreadakis, S and Raptis, E and Varsos, D},
  journal={Archiv der Mathematik},
  volume={47},
  number={1},
  pages={1--5},
  year={1986},
  publisher={Springer}
}

@article{Baumslag_Solitar,
  title={Some two-generator one-relator non-Hopfian groups},
  author={Baumslag, Gilbert and Solitar, Donald},
  journal={Bull. Amer. Math. Soc.},
  volume={68},
  number={6},
  pages={199--201},
  year={1962}
}

@article{Borisov_Sapir,
  title={Polynomial maps over finite fields and residual finiteness of mapping tori of group endomorphisms},
  author={Borisov, Alexander and Sapir, Mark},
  journal={Inventiones mathematicae},
  volume={160},
  number={2},
  pages={341--356},
  year={2005},
  publisher={Springer}
}

@article{Fujiwara_Sela,
  title={The rates of growth in a hyperbolic group},
  author={Fujiwara, Koji and Sela, Zlil},
  journal={Inventiones mathematicae},
  volume={233},
  number={3},
  pages={1427--1470},
  year={2023},
  publisher={Springer}
}

@incollection{Gromov,
  title={Hyperbolic groups},
  author={Gromov, Mikhael},
  booktitle={Essays in group theory},
  pages={75--263},
  year={1987},
  publisher={Springer}
}

@article{Higman,
  title={A finitely related group with an isomorphic proper factor group},
  author={Higman, Graham},
  journal={Journal of the London Mathematical Society},
  volume={1},
  number={1},
  pages={59--61},
  year={1951},
  publisher={Oxford University Press}
}

@article{Hirsch,
  title={On infinite soluble groups ({III})},
  author={Hirsch, K. A.},
  journal={Proceedings of the London Mathematical Society},
  volume={2},
  number={1},
  pages={184--194},
  year={1946}
}

@article{Hopf,
  author  = {Hopf, Heinz},
  title   = {Beitr{\"a}ge zur Klassifizierung der Fl{\"a}chenabbildungen},
  journal = {Journal f{\"u}r die reine und angewandte Mathematik},
  volume  = {165},
  year    = {1931},
  pages   = {225--236},
  doi     = {10.1515/crll.1931.165.225}
}

@article{Hsu_Wise,
  title={Ascending HNN extensions of polycyclic groups are residually finite},
  author={Hsu, Tim and Wise, Daniel T},
  journal={Journal of Pure and Applied Algebra},
  volume={182},
  number={1},
  pages={65--78},
  year={2003},
  publisher={Elsevier}
}

@article{Kim_Kyzy_Lee,
  title={A recipe for constructing non-Hopfian relatively hyperbolic groups with Hopfian peripheral subgroups},
  author={Kim, Jan and Kyzy, Tattybubu Arap and Lee, Donghi},
  journal={arXiv preprint arXiv:2306.14215},
  year={2023}
}

@article{Malcev,
  title={On isomorphic matrix representations of infinite groups},
  author={Mal'cev, Anatolii},
  journal={Matematicheskii Sbornik},
  volume={50},
  number={3},
  pages={405--422},
  year={1940},
  publisher={Russian Academy of Sciences, Steklov Mathematical Institute of Russian~…}
}

@article{Neumann,
  title={A two-generator group isomorphic to a proper factor group},
  author={Neumann, Bernhard Hermann},
  journal={Journal of the London Mathematical Society},
  volume={1},
  number={4},
  pages={247--248},
  year={1950},
  publisher={Oxford University Press}
}

@article{Osin,
  title={Relatively hyperbolic groups: Intrinsic geometry, algebraic properties, and algorithmic problems},
  author={Osin, Denis V},
  journal={Memoirs of the American Mathematical Society},
  volume={179},
  number={843},
  pages={1--106},
  year={2006},
  publisher={American Mathematical Society}
}

@article{Osin_2,
  title={Peripheral fillings of relatively hyperbolic groups},
  author={Osin, Denis V},
  journal={Inventiones mathematicae},
  volume={167},
  number={2},
  pages={295--326},
  year={2007},
  publisher={Springer}
}

@article{Osin_3,
  title={Elementary subgroups of relatively hyperbolic groups and bounded generation},
  author={Osin, Denis V},
  journal={International Journal of Algebra and Computation},
  volume={16},
  number={01},
  pages={99--118},
  year={2006},
  publisher={World Scientific}
}

@article{Reinfeldt_Weidmann,
  title={Makanin--Razborov diagrams for hyperbolic groups},
  author={Weidmann, Richard and Reinfeldt, Cornelius},
  journal={Annales Math{\'e}matiques Blaise Pascal},
  volume={26},
  number={2},
  pages={119--208},
  year={2019}
}

@article{Sapir_Wise,
  title={Ascending HNN extensions of residually finite groups can be non-Hopfian and can have very few finite quotients},
  author={Sapir, Mark and Wise, Daniel T},
  journal={Journal of Pure and Applied Algebra},
  volume={166},
  number={1-2},
  pages={191--202},
  year={2002},
  publisher={Elsevier}
}

@article{Tomar,
  title={Boundaries of graphs of relatively hyperbolic groups with cyclic edge groups},
  author={Tomar, Ravi},
  journal={Proceedings-Mathematical Sciences},
  volume={132},
  number={2},
  pages={47},
  year={2022},
  publisher={Springer}
}

\end{document}